\documentclass[11pt,a4paper, oneside]{amsart}

\usepackage{amsmath, amsthm, amssymb, amsfonts, enumerate, a4wide}
\usepackage{color}
\usepackage{geometry}
\usepackage{soul}
\usepackage{mathtools}

\usepackage[colorlinks=true,linkcolor=blue,urlcolor=blue, citecolor=blue]{hyperref}
\usepackage{bookmark}

\def \cA{\mathcal{A}}

\def \cF{\mathcal{F}}

\def \cJ{\mathcal{J}}

\def \cL{\mathcal{L}}

\def \cO{\mathcal{O}}

\def \cT{\mathcal{T}}
\def \P{\mathsf P}

\def \E{\mathsf E}

\def \N{\mathbb{N}}
\def \R{\mathbb{R}}
\def \F{\mathbb F}

\def \ud{\mathrm{d}}

\def\RdT{{\R^{d+1}_{0,T}}}

\newcommand{\roundbrackets}[1]{(#1)}

\newtheorem{theorem}{Theorem}[section]
\newtheorem{lemma}[theorem]{Lemma}
\newtheorem{corollary}[theorem]{Corollary}
\newtheorem{proposition}[theorem]{Proposition}

\newtheorem{remark}[theorem]{Remark}
\newtheorem{assumption}[theorem]{Assumption}

\theoremstyle{definition}

\makeatletter
\@namedef{subjclassname@2020}{\textup{2020} Mathematics Subject Classification}
\makeatother

\title[Stopper vs.\ singular controller games with degenerate diffusions]{Stopper vs. singular controller games \\ with degenerate diffusions}
\author[Bovo]{Andrea Bovo}
\author[De Angelis]{Tiziano De Angelis}
\author[Palczewski]{Jan Palczewski}
\address{A.\ Bovo: School of Management and Economics, Dept.\ ESOMAS, University of Torino, Corso Unione Sovietica, 218 Bis, 10134, Torino, Italy.}
\email{\href{mailto:mmab@leeds.ac.uk}{andrea.bovo@unito.it}}
\address{T.\ De Angelis: School of Management and Economics, Dept.\ ESOMAS, University of Torino, Corso Unione Sovietica, 218 Bis, 10134, Torino, Italy; Collegio Carlo Alberto, Piazza Arbarello 8, 10122, Torino, Italy.}
\email{\href{mailto:tiziano.deangelis@unito.it}{tiziano.deangelis@unito.it}}
\address{J.\ Palczewski: School of Mathematics, University of Leeds, Woodhouse Lane, LS2 9JT Leeds, UK.}\date{\today}
\email{\href{mailto:j.palczewski@leeds.ac.uk}{j.palczewski@leeds.ac.uk}}

\numberwithin{equation}{section}

\begin{document}

\begin{abstract}
We study zero-sum stochastic games between a singular controller and a stopper when the (state-dependent) diffusion matrix of the underlying controlled diffusion process is degenerate.  In particular, we show the existence of a value for the game and determine an optimal strategy for the stopper. The degeneracy of the dynamics prevents the use of analytical methods based on solution in Sobolev spaces of suitable variational problems. Therefore we adopt a probabilistic approach based on a perturbation of the underlying diffusion modulated by a parameter $\gamma>0$. For each $\gamma>0$ the approximating game is non-degenerate and admits a value $u^\gamma$ and an optimal strategy $\tau^\gamma_*$ for the stopper. Letting $\gamma\to 0$ we prove convergence of $u^\gamma$ to a function $v$, which identifies the value of the original game. We also construct explicitly optimal stopping times $\theta^\gamma_*$ for $u^\gamma$, related but not equal to $\tau^\gamma_*$, which converge almost surely to an optimal stopping time $\theta_*$ for the game with degenerate dynamics. 
\end{abstract}

\maketitle

\section{Introduction}

We consider stochastic zero-sum games between a singular controller and a stopper in a {\em degenerate} diffusive set-up. The underlying controlled dynamics is described by a stochastic differential equation (SDE) in $\R^d$ which is linearly affected by a singular control (i.e., controls with paths that are singular with respect to the Lebesgue measure). Differently from existing results, the state-dependent diffusion coefficient matrix $\sigma(x)$ of the SDE is not assumed to be {\em uniformly elliptic}. That means that for any $x\in\R^d$ there may be $\zeta\in\R^d$, $\zeta\neq 0$, such that $\langle \sigma\sigma^\top (x)\zeta,\zeta\rangle_d=0$, where $\sigma^\top$ is the transpose of $\sigma$ and $\langle\cdot,\cdot\rangle_d$ is the scalar product in $\R^d$. 

In \cite{bovo2022variational}, we started a systematic study of zero-sum singular-controller vs.\ stopper games in a diffusive set-up. Our approach in \cite{bovo2022variational} is based on a mixture of probabilistic methods and partial differential equations (PDE), which crucially relies on the assumption of uniform ellipticity of the diffusion coefficient. However, numerous irreversible (partially reversible) investment models for a single agent require more flexibility. For example, works by Zervos et al. \cite{lokka2011model,lokka2013long,merhi2007model}, Guo and Tomecek \cite{guo2009class}, Federico et al. \cite{federico2021singular,federico2014characterization}, Ferrari \cite{ferrari2015integral}, De Angelis et al.  \cite{de2017optimal,de2015nonconvex,deangelis2019solvable} consider a controlled process $X$ which is fully degenerate in the direction of the controlled coordinate (i.e., there is no diffusion in the control direction). It is therefore natural to consider similar set-ups in the context of stochastic games.

Removing the assumption of uniform ellipticity from the specification of the diffusion matrix makes an application of the PDE methods from \cite{bovo2022variational} no longer viable. The value function of the game in \cite{bovo2022variational} is obtained as a solution of a suitable (nonlinear) variational problem. The latter is solved by approximation via a family of penalised (semilinear) PDEs and employing compactness arguments in Sobolev spaces. Several bounds in Sobolev norms for the solutions of the penalised problems are required to guarantee compactness. Without uniform ellipticity such bounds can no longer be guaranteed (see, e.g., the proofs of \cite[Prop.\ 4.9, Prop.\ 5.1, Lem.\ 5.8]{bovo2022variational}).

In this paper, we obtain that the value of the game, $v:[0,T]\times\R^d\to \R$, exists and we find an optimal stopping time $\theta_*$ for the stopper. The latter is of a form slightly different to the one obtained in \cite{bovo2022variational} and commonly encountered in optimal stopping problems and stopping games (see the discussion following Theorem \ref{thm:usolvar} below): it is in the form of a hitting time for the pair given by the controlled process and its left limit. 

The methodology of our paper is based on an approximation of the game with games in which the controlled dynamics is non-degenerate. The perturbation to the original (degenerate) dynamics is modulated by a parameter $\gamma>0$ which vanishes in the limit of the approximation procedure. For each $\gamma$, we use results from \cite{bovo2022variational} to guarantee the existence of a value $u^\gamma:[0,T]\times\R^d\to \R$ for the associated game with non-degenerate dynamics $X^\gamma$ and of an optimal stopping time $\tau_*^\gamma$. Letting $\gamma\to 0$ we obtain the convergence $u^\gamma\to v$ uniformly in compact subsets of $[0,T]\times\R^d$. We also obtain convergence of explicitly constructed optimal stopping times $\theta^\gamma_*$ for $u^\gamma$, related but not equal to $\tau_*^\gamma$,
to the optimal stopping time $\theta_*$ in the original game.
We do not obtain a characterisation of $v$ in terms of a variational inequality because of the degeneracy of the diffusion coefficient. However, our approach enables us to allow a broader class of payoff functions than the ones considered in \cite{bovo2022variational}. The expected payoff of the game depends on functions $f,g,h$, that represent the cost of exerting control, the terminal payoff and the running payoff, respectively. In \cite{bovo2022variational}, these functions must be continuously differentiable in time and space, with H\"older-continuous derivatives on $[0,T]\times\R^d$ (this is required for the PDE methods to work). Here instead we only impose that $f,g,h$ be continuous and satisfy mild growth conditions in the spatial variable. Existence of the value $v$ and of an optimal stopping time $\theta_*$ under such weaker regularity conditions on $f,g,h$ can be obtained thanks to another approximation procedure, nested in the one required to deal with the degenerate dynamics.

The overall philosophy of this paper is close in spirit to the one in Bovo et al. \cite{bovo2023degcontrols}. In that paper we consider zero-sum stochastic games between a singular controller and a stopper, under a constraint on the directions of the controls (i.e., the controller can only control $d_0<d$ of the $d$ coordinates of the process). Differently from our setup, the approximation in \cite{bovo2023degcontrols} concerns the space of admissible controls while the diffusion coefficient matrix is uniformly elliptic. It is also worth noticing that in \cite[Ass.\ 2.1\roundbrackets{i}]{bovo2023degcontrols} additional restrictions are required on the structure of the diffusion coefficient matrix. Similar conditions are considered in our Assumption \ref{ass:gen1}(i.a), but we then show that they can be dropped when $f,g,h$ are sufficiently smooth (see Corollary \ref{cor:final}).  

The literature on stochastic games between singular controllers and stoppers is still in its early stages. For zero-sum games the field was initiated by Hernandez-Hernandez et al. \cite{hernandez2015zero}, \cite{hernandez2015zsgsingular} for problems with infinite-time horizon and one-dimensional controlled dynamics. In those papers, it is possible to construct explicit solutions in particular examples using an educated guess on the structure of the optimal controls of the two players. The method is enabled by the one-dimensional state-space, which leads to the study of ordinary differential equations (ODEs), rather than PDEs, and it does not extend to higher dimensional settings. Similar methods based on {\em guess-and-verify} approach have also been used by Ekstr\"om et al.\ \cite{ekstrom2023finetti} in a nonzero-sum singular-controller vs.\ stopper game with asymmetric information. We initiated the study of singular-controller vs.\ stopper zero-sum games in general diffusive setup in \cite{bovo2022variational} and then considered the problem with constrained control directions in \cite{bovo2023degcontrols}. The present paper continues that strand of the literature with the analysis of a degenerate diffusive setting. A broader literature review on stochastic games with {\em classical controls} and stopping times is provided in the introductions of \cite{bovo2023degcontrols} and \cite{bovo2022variational}. 

The paper is organised as follows. At the end of this introduction we provide some basic notation adopted throughout the paper. In Section \ref{sec:setting_3} we introduce the game in the degenerate diffusive set-up. We then state the main assumptions and the main results of the paper (Theorem \ref{thm:usolvar}). In Section \ref{sec:perturbed-dynamics} we introduce the approximating games with non-degenerate dynamics and smooth payoff functions. In Section \ref{sec:convergence} we prove convergence of the approximating games to the original one and, in particular, we show existence of the value and of an optimal stopping strategy for the original game. In Section \ref{sec:ext} we present some refinements and extensions of Theorem \ref{thm:usolvar} under different sets of assumptions.

\subsection{Notation}\label{sec:pre}
In what follows, $d,d'\in\N$ and $T\in(0,\infty)$. The Euclidean norm in $\R^d$ is denoted by $|\cdot|_d$ and the scalar product by $\langle \cdot, \cdot \rangle$. Given a matrix $M\in\R^{d\times d'}$ with entries $M_{ij}$, $i=1,\ldots d$, $j=1,\ldots d'$, its norm is given by
\begin{align*}
|M|_{d\times d'}\coloneqq \Big(\sum_{i=1}^d\sum_{j=1}^{d'}M_{ij}^2\Big)^{1/2}.
\end{align*}
For a square matrix $M \in \R^{d \times d}$, we write $\mathrm{tr}(M)\coloneqq \sum_{i=1}^d M_{ii}$ for its trace. Notice that
$|M|_{d\times d'}  = \big(\mathrm{tr}(M M^\top)\big)^{1/2}$, where $M^\top$ is the transpose of $M$.

The state space in our problem will be
\[
\R^{d+1}_{0,T}:=[0,T]\times\R^d.
\]
For a smooth function $f:\R^{d+1}_{0,T}\to \R$, we denote its partial derivatives by $\partial_t f$, $\partial_{x_i}f$, $\partial_{x_ix_j}f$,  $i,j=1,\ldots d$. The spatial gradient is defined as $\nabla f=(\partial_{x_1}f,\ldots \partial_{x_d} f)$, and $D^2 f$ denotes the spatial Hessian matrix with entries $\partial_{x_i x_j}f$ for $i,j=1,\ldots d$.

For an open set $D\subset \R^{d+1}_{0,T}$, let $C^\infty_{c, {\rm sp}}(D)$ be the space of real-valued functions on $D$ \emph{with compact support in the spatial coordinates} (not in time) and infinitely many continuous derivatives. 
For $p\in[1,\infty)$, $W^{1,2,p}_{\ell oc}(\R^{d+1}_{0,T})$ is the Sobolev space $\big\{ f\in L^p_{\ell oc}(\R^{d+1}_{0,T}) :\, f\in W^{1,2,p}(\cO), \;\forall\, \cO\subseteq \R^{d+1}_{0,T}, \cO\text{ open, bounded}\big\}$ (see \cite[Sec.\ 2.2]{krylov2008lectures}).

\section{Setting and Main Results}\label{sec:setting_3}
Our model has a finite horizon $T\in(0,\infty)$. Let $(\Omega,\cF,\P)$ be a complete probability space, $\F = (\cF_s)_{s\in[0,T]}$ a right-continuous filtration completed by the $\P$-null sets and $(W_s)_{s\in[0,T]}$ a $\F$-adapted, $d'$-dimensional Brownian motion.  There are two players engaged in a game. Player 1 (\emph{the stopper}) chooses a stopping time with respect to the filtration $\F$ at which the game is terminated. Player 2 (\emph{the controller}) chooses a singular control pair $(n,\nu)$, where $(n_t)_{t\in[0,T]}$ is $\F$-progressively measurable, $\R^d$-valued, such that $|n_t|_d=1$ for all $t\in[0,T]$, $\P$-a.s., and $(\nu_t)_{t\in[0,T]}$ is real-valued, non-decreasing, c\`adl\`ag with $\nu_{0-}=0$,\ $\P$-a.s. Such control pair modulates the dynamics of the underlying $d$-dimensional diffusion $(X_t^{[n,\nu]})_{t \in [0, T]}$ given by
\begin{align}\label{eq:SDE}
X_t^{[n,\nu]}= X^{[n,\nu]}_{0-} + \int_0^t b\big(X_s^{[n,\nu]}\big)\ud s + \int_0^t \sigma\big(X_s^{[n,\nu]}\big)\ud W_s + \int_{[0, t]} n_s\ud \nu_s,
\end{align}
where $b:\R^d\to\R^d$ and $\sigma:\R^{d}\to \R^{d}\times \R^{d'}$ are Lipschitz continuous functions. Notice that for $\P$-a.e.\ $\omega$, the map $s\mapsto n_s(\omega)$ is Borel-measurable on $[0,T]$ and $s\mapsto \nu_s(\omega)$ defines a measure on $[0,T]$; thus the Lebesgue-Stieltjes integral $\int_{[0,t]}n_u(\omega)\ud \nu_u(\omega)$ is well-defined for all $t \in [0, T]$ for $\P$-a.e.\ $\omega$. The value $X_{0-}^{[n, \nu]}$ denotes the initial state of the process before a possible shift via controls, i.e., $X^{[n,\nu]}_0 = X^{[n,\nu]}_{0-} + n_0 \Delta \nu_0$. 
We do not make any assumptions about the relationship between the dimension $d$ of the state process and the dimension $d'$ of the driving noise. This points to a distinguishing feature of our framework in which the uncontrolled dynamics of the state process can be \emph{degenerate} in the sense clarified below. 

We denote by $X^{[e_1,0]}$ the uncontrolled process, where $e_1$ is the unit vector in $\R^{d}$ with $1$ in the first entry. Its infinitesimal generator reads
\begin{align}\label{eq:cL}
(\cL\varphi)(x)=\frac{1}{2}\mathrm{tr}\left(a(x)D^2\varphi(x)\right)+\langle b(x),\nabla \varphi(x)\rangle,
\end{align}
with $a(x)\coloneqq(\sigma\sigma^\top)(x)$. The operator $\cL$ is {\em degenerate} at a point $x$ if there exist $\zeta\in\R^d$ ($\zeta\neq 0$) such that $\langle\zeta,a(x)\zeta\rangle=0$. 

We formally introduce the class of {\em admissible controls} for Player 2 as 
\begin{equation*}
 \cA_t\coloneqq \left\{(n,\nu)\left|\begin{aligned}&\text{$(n_s)_{s\in[0,T-t]}$ is progressively measurable, $\R^{d}$-valued,} \\ 
 &\text{with $|n_s|_{d}=1$, $\forall s\in[0,T-t]$, $\P$-a.s.};\\
 &\text{$(\nu_s)_{s\in[0,T-t]}$ is $\mathbb{F}$-adapted, real valued, non-decreasing and}\\
 &\,\text{right-continuous with $\nu_{0-}=0$, $\P\text{-a.s.}$, and $\E[|\nu_{T-t}|^2]<\infty$}
 \end{aligned} \right. \right\}, \qquad t \in [0, T].
\end{equation*}
Here, $t$ has the meaning of the time at which the game starts but, since the uncontrolled diffusion is time-homogeneous, it is convenient to consider a game that starts at time $0$  with time-horizon equal to $T-t$. Admissible controls for Player 1 are $\F$-stopping times from the class
\begin{align*}
\cT_t\coloneqq \left\{\tau:\, \tau\text{ is $\F$-stopping time},\ \tau\in[0, T-t]\right\}, \qquad t\in[0,T].
\end{align*}

Under our assumptions on the coefficients $b, \sigma$, for any admissible $(n, \nu) \in \cA_t$, there is a unique $\F$-adapted solution to \eqref{eq:SDE} on $[0, T-t]$, see, e.g., \cite[Thm.\ 2.5.7]{krylov1980controlled}. We indicate the initial point of $(X^{[n ,\nu]}_s)_{s \in [0, T-t]}$ by a subscript in the probability measure and in the expectation: 
\[
\P_x\big(\,\cdot\,\big)=\P\big(\,\cdot\,\big|X^{[n,\nu]}_{0-}=x\big)\quad\text{and}\quad\E_x\big[\,\cdot\,\big]=\E\big[\,\cdot\,\big|X^{[n,\nu]}_{0-}=x\big].
\]
This is only a notation as the probability space does not change. Notice also that the process $X^{[n,\nu]}$ need not be Markovian for arbitrary $(n,\nu)$.

We study a class of 2-player zero-sum games (ZSGs) between a stopper and a (singular) controller. The stopper is a maximiser and she picks $\tau\in\cT_t$. The controller is a minimiser and she chooses a pair $(n,\nu)\in\cA_t$. Given continuous functions $g,h:\R^{d+1}_{0,T}\to [0,\infty)$ and $f:[0,T]\to(0,\infty)$, a fixed discount rate $r\ge 0$ and $(t,x)\in\RdT$, the game's {\em expected} payoff reads
\begin{align}\label{eq:payoff}
\cJ_{t,x}(n,\nu,\tau)= \E_{x}\Big[e^{-r\tau}g(t\!+\!\tau,X_\tau^{[n,\nu]})\!+\!\int_0^{\tau}\! e^{-rs}h(t\!+\!s,X_s^{[n,\nu]})\,\ud s\!+\!\int_{[0,\tau]}\! e^{-rs}f(t\!+\!s)\,\ud \nu_s \Big].
\end{align}

For $(t,x) \in \RdT$, we define the \emph{lower} and \emph{upper} value of the game, respectively, by
\begin{align}\label{eq:lowuppvfnc_2}
\underline{v}(t,x)\coloneqq \adjustlimits\sup_{\tau\in \mathcal{T}_t}\inf_{(n,\nu)\in \mathcal{A}_t} \cJ_{t,x}(n,\nu,\tau)\quad\text{and}\quad\overline{v}(t,x)\coloneqq \adjustlimits\inf_{(n,\nu)\in \mathcal{A}_t}\sup_{\tau\in \mathcal{T}_t} \cJ_{t,x}(n,\nu,\tau).
\end{align}
Then $\underline{v}(t,x)\leq \overline{v}(t,x)$ and if the equality holds we say that the game \emph{admits a value}: 
\begin{align}\label{eq:valuegame}
v(t,x)\coloneqq \underline{v}(t,x)=\overline{v}(t,x).
\end{align}

The study of the above game and the variational characterisation of the value $v$ are hampered by the possible degeneracy of $\cL$. The PDE arguments from \cite{bovo2022variational} rest on the assumption of non-degeneracy of the underlying state process. 
Indeed, one cannot expect the value to be a strong solution (i.e., in the Sobolev class $W^{1,2,p}_{\ell oc}$) to the variational problem in \cite{bovo2022variational} associated to the game above. Specific technical difficulties arise, in particular, in obtaining $L^p$-bounds for the Hessian matrix of $v$. Nevertheless, in this paper we recover the existence of a value and the characterisation of an optimal strategy for the stopper under quite general assumptions allowing for the degeneracy of the dynamics of the state process.

We divide our assumptions into two groups. The first one concerns the dynamics of the state process and the second one the payoff functional $\cJ_{t,x}$.

\begin{assumption}[Controlled SDE]\label{ass:gen1}
The functions $b$ and $\sigma$ are continuously differentiable on $\R^d$ and Lipschitz with constant $D_1$, i.e.,
\begin{align}\label{eq:lipscondSDE}
|b(x)-b(y)|_d+|\sigma(x)-\sigma(y)|_{d\times d'}\leq D_1|x-y|_d, \qquad\text{for all $x,y\in\R^d$.}
\end{align}
At least one of the following two conditions holds:
\begin{itemize}
	\item[(i.a)] $\sigma_{ij}(x) = \sigma_{ij}(x_i)$ for $i=1, \ldots d$, $j=1, \ldots d'$;
	\item[(i.b)] There exists $D_2>0$ such that
\begin{align}\label{eq:sqrtgrowth}
|\sigma(x)|_{d\times d'}\leq D_2(1+|x|_d)^{\frac{1}{2}},\qquad\text{for all $x\in\R^d$.}
\end{align}
\end{itemize}
\end{assumption}
Notice that \eqref{eq:lipscondSDE} implies that there exists $D_3$ such that 
\begin{align}\label{eq:lingrowcff2}
|b(x)|_d+|\sigma(x)|_{d\times d'}\leq D_3(1+|x|_d), \quad\text{for all $x\in\R^d$}.
\end{align}

Conditions (i.a) or (i.b) enable a delicate argument in Lemma \ref{lem:stabL1res}, which establishes an $L^1$-bound on the controlled process $X^{[n, \nu]}$, uniformly in a sufficiently rich class of admissible controls. This class of controls consists of those $(n, \nu) \in \cA_t$ for which the first moment of $\nu_{T-t}$ is bounded linearly in $x$ (cf.\ Lemma \ref{lem:nubnd}). A classical $L^2$-bound of $X^{[n, \nu]}$ would involve the second moment of $\nu$, which we cannot control. 

\begin{assumption}[Functions $f,g,h$]\label{ass:gen2}
The functions $f:[0,T]\to(0,\infty)$, $g,h:\R^{d+1}_{0,T}\to[0,\infty)$ are continuous on their respective domains. Moreover, the following hold:
\begin{itemize}
	\item[  (i)] The function $f$ is non-increasing;
	\item[ (ii)] There exist constants $K_1\in(0,\infty)$ and $\beta\in[0,1)$ such that	
	\begin{align}\label{eq:ghbeta}
	0\leq g(t,x)+h(t,x)\leq K_1(1+|x|_d^\beta),\quad x\in\R^{d+1}_{0,T};
	\end{align}
	\item[(iii)] The function $g$ is Lipschitz in the spatial coordinates with a constant bounded by $f$. That is, $|\nabla g(t,x)|_d\leq f(t)$ for a.e. $(t,x)\in\R^{d+1}_{0,T}$.
\end{itemize}
\end{assumption}

The next theorem is the main result of the paper. It establishes the existence of the value of the game and its properties, alongside optimality of a stopping time for Player 1.
\begin{theorem}\label{thm:usolvar}
Under Assumptions \ref{ass:gen1} and \ref{ass:gen2}, the game described above admits a value $v$ (i.e., \eqref{eq:valuegame} holds) with the following properties:
\begin{itemize}
	\item[(i)] $v$ is continuous on $\R^{d+1}_{0,T}$,
	\item[(ii)] $|v(t,x)|\leq c(1+|x|_d^\beta)$ for $\beta\in(0,1)$ from Assumption \ref{ass:gen2}(ii) and some $c>0$,\smallskip
	\item[(iii)] $v$ is Lipschitz continuous in space with constant bounded by $f$, i.e., $|\nabla v(t,x)|_d\leq f(t)$ for a.e. $(t,x)\in\R^{d+1}_{0,T}$.
\end{itemize}

For any given $(t,x)\in\R^{d+1}_{0,T}$ and any admissible control $(n,\nu)\in\cA_t$, let $\theta_*=\theta_*(t,x;n,\nu)\in\cT_t$ be defined as $\theta_*\coloneqq\tau_*\wedge\sigma_*$, where $\P_x$-a.s.\footnote{{Notice that $\theta_*\in[0,T-t]$, $\P$-a.s. since $\underline{v}(T, x) = \overline{v} (T,x) = g(T,x)$ due to Assumption \ref{ass:gen2}(iii).}}
\begin{align}\label{eq:taustar}
\begin{split}
&\tau_*\coloneqq \inf\big\{s\geq 0:\, v(t+s,X_s^{[n,\nu]})=g(t+s,X_s^{[n,\nu]})\big\},\\
&\sigma_*\coloneqq \inf\big\{s\geq 0:\, v(t+s,X_{s-}^{[n,\nu]})=g(t+s,X_{s-}^{[n,\nu]})\big\}.
\end{split}
\end{align}
Then, $\theta_*$ is optimal for the stopper in the sense that
\begin{align*}
v(t,x)=\inf_{(n,\nu)\in\cA_t}\cJ_{t,x}\big(n,\nu,\theta_*(t, x; n,\nu)\big),\quad (t,x)\in\RdT.
\end{align*}
\end{theorem}
The above theorem asserts that the value $v$ is continuous as a function of all its arguments and it is Lipschitz continuous in the spatial variable. The bound by $f$ of the norm of the spatial gradient of $v$ reflects the cost of action for the controller and is natural in singular control problems: a larger norm of $\nabla v$ is prevented by the possibility of the controller to exert an immediate shift at a cost $f$.

We identify an optimal strategy of the stopper, $\theta_*$, in terms of hitting times of the process $(t+s, X^{[n, \nu]}_s)_{s \in [0, T-t]}$ and of its left limits. The form of the strategy is unusual. One would expect the stopper to act according to $\tau_*$ akin to the classical optimal stopping literature and in agreement with \cite[Thm.\ 3.3]{bovo2022variational}. It turns out that $\tau_*$ is insufficient to guarantee optimality and the stopper needs to monitor the left limits of the state process too. We first note that as long as the controller acts continuously, the stopping times $\tau_*$ and $\sigma_*$ coincide. This suggests that under the current assumptions we cannot assert that it is suboptimal for the controller to shift the process exactly at $\sigma_*$ and, as a result, delay $\tau_*$.

The underlying idea of the proof of Theorem \ref{thm:usolvar} is to approximate the upper and lower value of our game by values of games with non-degenerate dynamics and smooth, compactly supported payoff functions. Those games are studied with PDE methods developed in \cite{bovo2022variational}. A sequence of delicate limiting arguments are then employed to remove the random perturbation which brings in the non-degeneracy and, subsequently, to relax assumptions on the payoff functions. Details are presented in Sections \ref{sec:perturbed-dynamics}-\ref{sec:convergence}. In Section \ref{sec:ext} we further show how to relax assumptions on the payoff functions beyond Assumption \ref{ass:gen2} and assumptions on the diffusion coefficient beyond Assumption \ref{ass:gen1}(i.a) and (i.b). In particular, we can allow for more general growth conditions on the payoff functions at the cost of imposing higher smoothness thereof.

Before turning our attention to technical arguments, we remark that with no loss of generality we can restrict Player 2's admissible controls to those with bounded expectation (linearly in $x$). The proof is similar to \cite[Lem.~3.1]{bovo2023degcontrols} and therefore omitted.
\begin{lemma}\label{lem:nubnd}
There is a constant $K_2>0$ such that for any $(t,x)\in\R^{d+1}_{0,T}$ we have
\begin{align*}
&\underline{v}(t,x)=\adjustlimits\inf_{(n,\nu)\in\cA^{opt}_{t,x}}\sup_{\tau\in\cT_t}\cJ_{t,x}(n,\nu,\tau),\\
&\overline{v}(t,x)=\adjustlimits\sup_{\tau\in\cT_t}\inf_{(n,\nu)\in\cA^{opt}_{t,x}}\cJ_{t,x}(n,\nu,\tau),
\end{align*}
where $\cA^{opt}_{t,x}\coloneqq\{(n,\nu)\in\cA_t:\, \E_x[\nu_{T-t}]\leq K_2(1+|x|_d)\}$. The constant $K_2$ depends on $D_3$ in \eqref{eq:lingrowcff2}, $K_1$ in \eqref{eq:ghbeta}, $T$, $d$ and $f(T)$.
\end{lemma}

Notice that the original class of admissible controls $\cA_t$ does not depend on $x$, whereas $\cA_{t,x}^{opt}$ does.

\section{Approximating games with a perturbed controlled dynamics}\label{sec:perturbed-dynamics}

In this section, we introduce a family of ZSGs with non-degenerate dynamics which will be shown in Section \ref{sec:convergence} to approximate the game in our paper. We will first prove Theorem \ref{thm:usolvar} under stronger assumptions on the payoff functions. These conditions will be relaxed in Section \ref{sec:ass3.2}. Unless stated otherwise, we will now proceed under the following conditions.
\begin{assumption}\label{ass:gen3}
The functions $f:[0,T]\to(0,\infty)$, $g,h:\R^{d+1}_{0,T}\to[0,\infty)$ are such that
\begin{itemize}
\item[(i)] $g, h\in C^{\infty}_{c,{\rm sp}}(\R^{d+1}_{0,T})$;
\item[(ii)] $f\in C^{\infty}([0,T])$ and non-increasing;
\item[(iii)] $|\nabla g(t,x)|_{d}\leq f(t)$ for all $(t,x)\in\R^{d+1}_{0,T}$.
\end{itemize}
\end{assumption}
Since $g,h$ are smooth and compactly supported, there is a constant $K \in (0, \infty)$ such that
\begin{itemize}
\item[(iv)] $f,g$ and $h$ are bounded and, for all $0\leq s<t\leq T$ and  $x,y\in \R^{d}$,
\begin{align}\label{eq:ghsmooth}
|g(t,x)-g(s,y)|+|h(t,x)-h(s,y)|\leq K\big(|x-y|_d+(t-s)\big);
\end{align}
\item[(v)] for all $(t,x)\in\R^{d+1}_{0,T}$
\[
(h+\partial_tg+\cL g-rg)(t,x)\geq -K.
\]
\end{itemize}
Therefore, Assumption \ref{ass:gen1} together with Assumption \ref{ass:gen3} imply \cite[Assumptions 3.1 and 3.2]{bovo2022variational} except for the non-degeneracy of the diffusion coefficient $\sigma$. To address the latter issue, we fix $\gamma\in(0,1)$ and, given $(n,\nu)\in\cA_t$, we consider a \emph{perturbed dynamics} of the state process:
\begin{align}\label{eq:SDEgamma}
\ud X_{s}^{[n,\nu],\gamma}=&\, b(X_s^{[n,\nu],\gamma})\,\ud s +\sigma(X_s^{[n,\nu],\gamma})\ud W_s+\gamma \ud \widetilde{W}_s+ n_{s}\,\ud\nu_s,
\end{align}
where $(\widetilde{W}_s)_{s \ge 0}$ is a $d$-dimensional Brownian motion independent from the Brownian motion $(W_s)_{s \ge 0}$. We denote by $\cJ^\gamma_{t,x}$ the payoff $\cJ_{t,x}$ defined in \eqref{eq:payoff} with $X^{[n,\nu]}$ replaced by $X^{[n,\nu],\gamma}$, i.e.,
\begin{align*}
\cJ_{t,x}^{\gamma}(n,\nu,\tau)= \E_{x}\Big[&\,e^{-r\tau}g(t+\tau,X_\tau^{[n,\nu],\gamma})+\int_0^{\tau}\! e^{-rs}h(t+s,X_s^{[n,\nu],\gamma})\,\ud s+\int_{[0,\tau]}\! e^{-rs}f(t+s)\,\ud \nu_s \Big].
\end{align*}
We say that the game with expected payoff $\cJ^\gamma_{t,x}$ admits a value if
\begin{align}\label{eq:gametheta}
u^\gamma(t,x)=\adjustlimits\sup_{\tau\in\cT_t}\inf_{(n,\nu)\in\cA_t}\cJ_{t,x}^{\gamma}(n,\nu,\tau)= \adjustlimits\inf_{(n,\nu)\in\cA_t}\sup_{\tau\in\cT_t}\cJ_{t,x}^{\gamma}(n,\nu,\tau).
\end{align}

With the addition of the perturbation term $\gamma \ud \widetilde{W}_s$ in the state dynamics, the associated differential operator $\cL^\gamma$ becomes $(\cL^\gamma \varphi)(x)=\frac{1}{2}\mathrm{tr}\left(a_\gamma(x)D^2\varphi(x)\right)+\langle b(x),\nabla \varphi(x)\rangle$, where $a_\gamma(x)=a(x)+\gamma^2\mathrm{I}_d$, with $\mathrm{I}_d$ denoting the $d$-dimensional identity matrix (cf.~\eqref{eq:cL} for the original dynamics). The operator $\cL^\gamma$ is uniformly non-degenerate: 
\begin{align*}
\langle \zeta,a_\gamma(x)\zeta \rangle=\,\langle \zeta,a(x)\zeta\rangle+\gamma^2|\zeta|_d^2 \geq \gamma^2|\zeta|_d^2,
\end{align*}
for all $x,\zeta\in\R^d$. Thus, under Assumptions \ref{ass:gen1} and \ref{ass:gen3}, the game with payoff $\cJ^\gamma_{t,x}$ satisfies all assumptions of \cite[Thm.\ 3.3]{bovo2022variational}. We reproduce below its main assertions.
\begin{theorem}\label{thm:from3to5}
The game with payoff $\cJ^\gamma_{t,x}$ admits a value (i.e., \eqref{eq:gametheta} holds). The value function $u^\gamma$ is the maximal solution in the class $W^{1,2,p}_{\ell oc}(\R^{d+1}_{0,T})$, for arbitrary $p\in[1,\infty)$, of the variational inequality 
\begin{align}\label{eq:varineq}
\begin{split}
&\min\big\{\max\big\{ \partial_tu+\cL^\gamma u-ru+h,g-u\big\},f-|\nabla u|_d\big\}=0,\\
&\max\big\{\min\big\{ \partial_tu+\cL^\gamma u-ru+h,f-|\nabla u|_d\big\},g-u\big\}=0,
\end{split}
\end{align}
with $u(T,x)=g(T,x)$ and growth condition $|u(t,x)|\leq c(1+|x|_d)$, for a suitable $c>0$.

For any given $(t,x)\in\R^{d+1}_{0,T}$ and any admissible control $(n,\nu)\in\cA_t$, the stopping time defined under $\P_x$ as
\begin{align}\label{eq:taustar_2}
\tau_*^\gamma=\tau_*^\gamma(t,x;n,\nu)\coloneqq \inf\big\{s\geq 0:\, u^\gamma(t+s,X_s^{[n,\nu],\gamma})=g(t+s,X_s^{[n,\nu],\gamma})\big\}
\end{align}
is optimal for the stopper.
\end{theorem}

Thanks to the boundedness and positivity of $f,g,h$, the value function of the game $u^\gamma$ is bounded. Indeed, it is non-negative and the upper bound follows by taking a sub-optimal control $(n,\nu)=(e_1,0)$ with $e_1=(1,0,\ldots 0) \in \R^d$.

As in Lemma \ref{lem:nubnd}, using the same arguments of proof as in \cite[Lem.~3.1]{bovo2023degcontrols}, we can restrict our attention to a subset of admissible controls which is the same for any $\gamma \in(0, 1)$. The latter property is important when we study the behaviour of the game as $\gamma \to 0$.
\begin{lemma}\label{lem:L1bnd}
For any $(t,x)\in\R^{d+1}_{0,T}$, we have
\begin{align*}
u^\gamma(t,x)=\inf_{(n,\nu)\in\cA^{opt}_{t,x}}\sup_{\tau\in\cT_t}\cJ_{t,x}^\gamma(n,\nu,\tau)=\sup_{\tau\in\cT_t}\inf_{(n,\nu)\in\cA^{opt}_{t,x}}\cJ_{t,x}^\gamma(n,\nu,\tau),
\end{align*}
where $\cA_{t,x}^{opt}\coloneqq\big\{(n,\nu)\in\cA_t:\,\E_x[\nu_{T-t}]\leq K_2(1+|x|_d)\big\}$ and the constant $K_2 > 0$ can be chosen the same here and in Lemma \ref{lem:nubnd}, i.e., $K_2 = K_2(d, D_3, K_1, T, f(T))$ with $D_3$ from \eqref{eq:lingrowcff2} and $K_1$ from \eqref{eq:ghbeta}, and independent from $\gamma$.
\end{lemma}

Having established the properties of the approximating game with the payoff $\cJ^\gamma_{t,x}$, we turn our attention to the convergence of the functions $u^\gamma$ to the value of the original game (cf.\ \eqref{eq:valuegame}) when we let $\gamma\to 0$, and to the limiting behaviour of the optimal stopping times. It turns out that the limit of the family  $(\tau_*^\gamma)_{\gamma>0}$ as $\gamma\to 0$ may not be optimal for the stopper in the original game, see the discussion following Theorem \ref{thm:usolvar}. Instead, we construct a family of stopping times, $(\theta_*^\gamma)_{\gamma > 0}$, as follows. For $\gamma>0$, we define
\[
\theta_*^\gamma\coloneqq\tau_*^\gamma\wedge\sigma_*^\gamma,
\]
where, $\P_x$-a.s., 
\begin{align*}
\sigma_*^\gamma = \sigma_*^\gamma(t,x; n, \nu) \coloneqq \inf\big\{s\geq 0:\, u^\gamma(t+s,X_{s-}^{[n,\nu],\gamma})=g(t+s,X_{s-}^{[n,\nu],\gamma})\big\}.
\end{align*}
The following lemma demonstrates that the stopping time $\theta_*^\gamma=\theta_*^\gamma(t,x; n,\nu)$ is also optimal for the stopper in the game with value $u^\gamma$.

\begin{lemma}\label{lem:theta*opt}
For any $(t,x)\in\R^{d+1}_{0,T}$,  we have
\begin{align}
u^{\gamma}(t,x)=\inf_{(n,\nu)\in\cA_t}\cJ_{t,x}^\gamma\big(n,\nu,\theta_*^\gamma(t, x; n,\nu)\big),
\end{align}
hence $\theta_*^\gamma$ is optimal for the stopper in the game with value $u^\gamma$.
\end{lemma}

Notice that the stopping time $\theta_*^\gamma$ is of feedback form, i.e., it depends on the dynamics of the state process and, via this dynamics, on the initial point $(t, x)$ and on the control $(n, \nu)$ applied by the controller.

The proof of Lemma \ref{lem:theta*opt} is similar to the proof of \cite[Lem.\ 3.5]{bovo2023degcontrols} and therefore it is omitted. We only provide an intuitive justification. Note that $\sigma_*^\gamma < \tau_*^\gamma$ only if the controller shifts the state process in a discontinuous way at $\sigma_*^\gamma$ and $u^\gamma(t+\sigma_*^\gamma,X_{\sigma_*^\gamma}^{[n,\nu],\gamma})>g(t+\sigma_*^\gamma,X_{\sigma_*^\gamma}^{[n,\nu],\gamma})$. By the continuity of $u^\gamma$ and $g$, we have $u^\gamma(t+\sigma_*^\gamma,X_{\sigma_*^\gamma-}^{[n,\nu],\gamma})=g(t+\sigma_*^\gamma,X_{\sigma_*^\gamma-}^{[n,\nu],\gamma})$, which implies
\begin{equation*}
\begin{aligned}
u^\gamma(t+\sigma_*^\gamma,X_{\sigma_*^\gamma}^{[n,\nu],\gamma})
- u^\gamma(t+\sigma_*^\gamma,X_{\sigma_*^\gamma-}^{[n,\nu],\gamma}) 
&>
g(t+\sigma_*^\gamma,X_{\sigma_*^\gamma}^{[n,\nu],\gamma})
- g(t+\sigma_*^\gamma,X_{\sigma_*^\gamma-}^{[n,\nu],\gamma})\\
&\ge -f(t + \sigma_*^\gamma) \big|X_{\sigma_*^\gamma}^{[n,\nu],\gamma} - X_{\sigma_*^\gamma-}^{[n,\nu],\gamma} \big|_d,
\end{aligned}
\end{equation*}
where the last inequality follows from Assumption \ref{ass:gen3}(iii). We rewrite it as
\[
u^\gamma(t+\sigma_*^\gamma,X_{\sigma_*^\gamma-}^{[n,\nu],\gamma}) < f(t + \sigma_*^\gamma) \big|X_{\sigma_*^\gamma}^{[n,\nu],\gamma} - X_{\sigma_*^\gamma-}^{[n,\nu],\gamma} \big|_d + u^\gamma(t+\sigma_*^\gamma,X_{\sigma_*^\gamma}^{[n,\nu],\gamma}).
\]
Comparing with the functional $\cJ^\gamma_{t,x}$ and recalling that $u^\gamma$ is the value function, this means that the controller who is a minimiser made a mistake of exerting a jump control at $\sigma_*^\gamma$: the jump increases the value of the game, hence it is against the controller's own interest. It would have been strictly better to control continuously at this time in which case we would have $\sigma_*^\gamma = \tau_*^\gamma$ and the payoff would be equal to $\cJ^\gamma_{t,x}(n, \nu, \sigma_*^\gamma)$. This shows that stopping at $\theta_*^\gamma$ gives at least the value $u^\gamma(t,x)$. However, for a specific choice of $(n, \nu)$ such that $\sigma_*^\gamma < \tau_*^\gamma$ with a positive probability, we have
\[
u^\gamma(t,x) \le \cJ^\gamma_{t,x} (n, \nu, \theta_*^\gamma) < \cJ^\gamma_{t,x} (n, \nu, \tau_*^\gamma).
\]
We can also conclude that in equilibrium (if it exists), the controller never shifts the state process discontinuously at $\sigma^\gamma_*$, in which case $\tau_*^\gamma = \sigma_*^\gamma = \theta_*^\gamma$.

\section{Convergence of the approximating games}\label{sec:convergence}
In this section we first study the limit as $\gamma\to0$ and then we relax the additional regularity conditions of Assumption \ref{ass:gen3}.

\subsection{Properties of the perturbed process}
We start from examination of properties of the perturbed process $X^{[n,\nu],\gamma}$ and its convergence to the unperturbed process $X^{[n,\nu]}$ as $\gamma\to 0$, for any fixed pair $(n,\nu)\in\cA_t$.
\begin{proposition}\label{prop:Lipschitzproc}
Fix $(t,x)\in\R^{d+1}_{0,T}$ and $p\in[1,\infty)$. For any $(n,\nu)\in\cA_t$, we have 
\begin{align}\label{eq:convtheta0l}
\E_x\Big[\sup_{s\in[0,T-t]}\big|X_s^{[n,\nu],\gamma}-X_s^{[n,\nu]}\big|_d^{p}\Big]\leq K_3 \gamma^{p}  ,
\end{align}
where $K_3=K_3(D_1,d,T,p)>0$ with $D_1$ from Assumption \ref{ass:gen1}. 
\end{proposition}

\begin{proof}
We start by proving the result for $p\geq2$. Take $(n,\nu)\in\cA_t$ and let $X^{[n,\nu]}$ and $X^{[n,\nu],\gamma}$ be the processes from \eqref{eq:SDE} and \eqref{eq:SDEgamma}, respectively. Let $Y^\gamma_s=X_s^{[n,\nu],\gamma}-X_s^{[n,\nu]}$. We have that for all $s\in[0,T-t]$
\begin{align*}
Y_s^\gamma=\int_0^s\big(b(X^{[n,\nu],\gamma}_r)-b(X^{[n,\nu]}_r)\big)\,\ud r+\gamma \widetilde{W}_s+\int_0^s\big(\sigma(X^{[n,\nu],\gamma}_r)-\sigma(X^{[n,\nu]}_r)\big)\,\ud W_r.
\end{align*}
Using the inequality $\big(\sum_{i=1}^k y_i\big)^p\leq k^{\,p-1}\big(\sum_{i=1}^k|y_i|^p\big)$, taking the supremum first and then the expectation, we get
\begin{align}\label{eq:lipcontheta0_2}
\E_x\Big[\sup_{\lambda\in[0,s]}|Y^\gamma_\lambda|_d^{p}\Big]\leq3^{p-1} \E_x\Big[&\sup_{\lambda\in[0,s]}\Big(\Big|\int_0^\lambda\!\big(b(X_{r}^{[n,\nu],\gamma})-b(X_{r}^{[n,\nu]})\big)\,\ud r\Big|_d^{p}+|\gamma \widetilde{W}_\lambda|_{d}^{p}\\
&+\Big|\int_0^\lambda\!\big(\sigma(X_{r}^{[n,\nu],\gamma})-\sigma(X_{r}^{[n,\nu]})\big)\,\ud W_r\Big|_d^{p}\Big)\Big].\notag
\end{align}

The first term on the right-hand side of \eqref{eq:lipcontheta0_2} is bounded from above using H\"older's inequality and the Lipschitz property of $b$ (see \eqref{eq:lipscondSDE})
\begin{align*}
&\E_x\Big[\sup_{\lambda\in[0,s]}\Big|\int_0^\lambda\!\big(b(X_{r}^{[n,\nu],\gamma})-b(X_{r}^{[n,\nu]})\big)\,\ud r\Big|_d^{p}\Big]
\leq \E_x\Big[ \Big(\int_0^s\!\big|b(X_{r}^{[n,\nu],\gamma})-b(X_{r}^{[n,\nu]})\big|_d\,\ud r \Big)^p \Big] \\
&\leq \E_x\Big[T^{p-1}\!\int_0^s\!\big|b(X_{r}^{[n,\nu],\gamma})-b(X_{r}^{[n,\nu]})\big|_d^{p}\,\ud r\Big]
\leq T^{p-1}D_1^{p}\, \E_x\Big[\int_0^s\!\sup_{\lambda\in[0,r]}\big|Y^\gamma_\lambda\big|_d^{p}\,\ud r\Big].
\end{align*}
The second term of the right-hand side of \eqref{eq:lipcontheta0_2} is bounded from above using the Doob's maximal inequality applied to the submartingale $(|\widetilde{W}_t|_d)_{t \ge 0}$
\begin{align*}
\E_{x}\Big[\sup_{\lambda\in[0,s]}\big|\gamma \widetilde{W}_\lambda\big|_{d}^{p}\Big]
\leq 
\gamma^p \big(\tfrac{p}{p-1}\big)^p \E_x\Big[\big|\widetilde{W}_T\big|_d^{p}\Big]
=: \kappa_1\gamma^{p},
\end{align*}
with $\kappa_1 = \kappa_1 (d, p, T)$. The last term on the right-hand side of \eqref{eq:lipcontheta0_2} is bounded from above using \cite[Cor.\ 2.5.11]{krylov1980controlled} and Lipschitz continuity of $\sigma$
\begin{align*}
\E_x\Big[&\sup_{\lambda\in[0,s]}\Big|\int_0^\lambda\!\big(\sigma(X_{r}^{[n,\nu],\gamma})-\sigma(X_{r}^{[n,\nu]})\big)\,\ud W_r\Big|_d^{p}\Big]\\
&\,\leq \kappa_2\E_x\Big[\int_0^s\!\big|\sigma(X_{r}^{[n,\nu],\gamma})-\sigma(X_{r}^{[n,\nu]})\big|_{d\times d'}^{p}\,\ud r\Big]
\leq \kappa_2 D_1^{p}\, \E_x\Big[\int_0^s\!\sup_{\lambda\in[0,r]}\big|Y^\gamma_\lambda\big|_d^{p}\,\ud r\Big],
\end{align*}
where $\kappa_2=\kappa_2(T,p)=2^{\frac{4+p}{2}}(p-1)^{\frac{p}{2}}T^{\frac{p}{2}-1}$ and $D_1$ comes from \eqref{eq:lipscondSDE}. 

We insert the above three bounds into \eqref{eq:lipcontheta0_2} and change the order of integration:
\[
\E_x\Big[\sup_{\lambda\in[0,s]}\big|Y^\gamma_\lambda\big|_d^{p}\Big]
\leq
3^{p-1} \Big( D_1^{p}(T^{p-1}+\kappa_2)\int_0^{s} \E_x\Big[\sup_{\lambda\in[0,r]}\big|Y^\gamma_\lambda\big|_d^{p}\Big] \,\ud r+\kappa_1\gamma^{p}\Big).
\]
This allows us to apply Gronwall's lemma and obtain the estimate
\begin{align*}
\E_x\Big[\sup_{\lambda\in[0,s]}\big|X_{\lambda}^{[n,\nu],\gamma}-X_{\lambda}^{[n,\nu]}\big|^p_d\Big]=\E_x\Big[\sup_{\lambda\in[0,s]}\big|Y_{\lambda}^{\gamma}\big|^p_d\Big]\leq K_3 \gamma^{p},\qquad s\in[0,T-t],
\end{align*}
with $K_3= K_3(D_1,d,p,T)$ independent of $(n,\nu)\in\cA_t$. 

Now, take $p \in [1,2)$. By Jensen's inequality we get
\begin{align*}
\E_x\Big[\sup_{\lambda\in[0,s]}\big|X_{\lambda}^{[n,\nu],\gamma}-X_{\lambda}^{[n,\nu]}\big|_d^p\Big]
&\leq 
\Big(\E_x\Big[\sup_{\lambda\in[0,s]}\big|X_{\lambda}^{[n,\nu],\gamma}-X_{\lambda}^{[n,\nu]}\big|_d^2\Big]\Big)^{p/2}\\
&\leq \big(\gamma^{2} K_3(D_1, d, 2, T)\big)^{p/2}=\gamma^p K_3(D_1, d, 2, T)^{p/2},
\end{align*}
with $K_3(D_1, d, p, T) =K_3(D_1,d,2,T)^{p/2}$.
\end{proof}

Under our assumptions we can guarantee an $L^1$-bound on the controlled dynamics, uniformly over the class of admissible controls. That will be useful later on when relaxing Assumption \ref{ass:gen3}.

\begin{lemma}\label{lem:stabL1res}
Under Assumption \ref{ass:gen1}, there exists a constant $K_4>0$ such that for arbitrary $(t,x)\in\R^{d+1}_{0,T}$, $(n,\nu)\in\cA_{t,x}^{opt}$ and a stopping time $\tau\in\cT_t$
\begin{align}\label{eq:XgamK4}
\E_x\big[\big|X_{\tau}^{[n,\nu],\gamma}\big|_d\big]\leq K_4(1+|x|_d), \qquad \gamma \in [0,1).
\end{align}
The constant $K_4$ depends on $d, D_1, D_3, K_2, T, f(T)$ and, in the case of Assumption \ref{ass:gen1}(i.b), also on $D_2$.
\end{lemma}

\begin{proof}
The proof is slightly different, depending on whether condition (i.a) or condition (i.b) in Assumption \ref{ass:gen1} hold. Under condition (i.a), we can use an argument analogous to the one in \cite[Cor.\ 3.9]{bovo2023degcontrols}. We first notice that 
\[
\E_x\big[\big|X_\tau^{[n,\nu],\gamma}\big|_d\big]\le \E_x\big[\big|X_\tau^{[n,\nu],\gamma}-X_\tau^{[e_1,0],\gamma}\big|_d\big]+\E_x\big[\big|X_\tau^{[e_1,0],\gamma}\big|_d\big].
\]
We bound the first term using \cite[Lem.\ 3.8]{bovo2023degcontrols} (with $d_0 = d$; here we need condition (i.a)) as follows
\[
\E_x\big[\big|X_\tau^{[n,\nu],\gamma}-X_\tau^{[e_1,0],\gamma}\big|_d\big] \le \kappa\, \E_x [\nu_{T-t}] \le c (1 + |x|_d),
\]
for some constant $c > 0$, where $\kappa = \kappa (d, D_1, T)$ is independent from $\gamma \in (0,1)$ and the last inequality follows from the definition of $\cA^{opt}_{t,x}$ in Lemma \ref{lem:L1bnd}. For the uncontrolled dynamics $X^{[e_1,0],\gamma}$, standard SDE estimates (e.g., \cite[Cor.\ 2.5.12]{krylov1980controlled}) yield
\[
\E_x\big[\big|X_\tau^{[e_1,0],\gamma}\big|_d\big]\le c(1+|x|_d),
\]
for some constant $c = c(d, D_1, D_3, T)>0$ independent from $\gamma\in(0,1)$. This completes the proof of \eqref{eq:XgamK4}.

Assume now that condition (i.b) in Assumption \ref{ass:gen1} holds. To shorten the notation, we will write $X^{\gamma}$ for $X^{[n,\nu],\gamma}$. From the dynamics of $X^\gamma$ we have 
\begin{equation}\label{eq:L1normtheta}
\begin{aligned}
&\E_x\Big[\sup_{\lambda\in[0,s]}|X_{\lambda}^{\gamma}|_d\Big]\\
&\leq \E_x\Big[\sup_{\lambda\in[0,s]}\Big(\Big|\int_0^\lambda\!\!b(X_{r}^{\gamma})\,\ud r\Big|_d\!+\Big|\int_0^\lambda\!\!\sigma(X_{r}^{\gamma})\,\ud W_r\Big|_d\!+\big|\gamma\widetilde{W}_\lambda\big|_d\!+\Big|\int_0^\lambda n_r\ud\nu_r\Big|_d\Big)\Big].
\end{aligned}
\end{equation}
We estimate each term individually. Thanks to linear growth of the drift, the first term on the right-hand side of \eqref{eq:L1normtheta} can be bounded as follows
\begin{align*}
\E_x\Big[\sup_{\lambda\in[0,s]}\Big|\int_0^\lambda\!b(X_{r}^{\gamma})\,\ud r\Big|_d\Big]
&\leq
\E_x\Big[D_3\int_0^{s}\!\big(1+\sup_{r\in[0,\lambda]}|X_{r}^{\gamma}|_d\big)\,\ud \lambda\Big]\\ 
&\leq 
D_3\Big(T+\E_x\Big[\int_0^{s}\!\sup_{r\in[0,\lambda]}|X_{r}^{\gamma}|_d\,\ud \lambda\Big]\Big).
\end{align*}
For the second term on the right-hand side of \eqref{eq:L1normtheta} we use Jensen's inequality and Doob's maximal inequality, taking advantage of the square-root growth of the diffusion coefficient, to obtain 
\begin{align*}
&\E_x\Big[\sup_{\lambda\in[0,s]}\Big|\int_0^\lambda\!\sigma(X_{r}^{\gamma})\,\ud W_r\Big|_d\Big]\\
&\leq\E_x\Big[\sup_{\lambda\in[0,s]}\Big|\int_0^\lambda\!\sigma(X_{r}^{\gamma})\,\ud W_r\Big|^2_d\Big]^{\frac12}\leq
2\E_x\Big[\Big|\int_0^s\!\sigma(X_{r}^{\gamma})\,\ud W_r\Big|_d^2\Big]^{\frac{1}{2}}\\
&=
2\E_x\Big[\int_0^{s}\!\big|\sigma(X_{r}^{\gamma})\big|^2_{d\times d}\,\ud r\Big]^{\frac{1}{2}}\leq
2\Big(1+\E_x\Big[\int_0^{s}\!\big|\sigma(X_{r}^{\gamma})\big|^2_{d\times d}\,\ud r\Big]\Big)\\
&\leq
2\Big(1+D_2\Big(T+\E_x\Big[\int_0^{s}\!\sup_{r\in[0,\lambda]}\big|X_{r}^{\gamma}\big|_d\,\ud \lambda\Big]\Big)\Big),
\end{align*}
where the third inequality uses $\sqrt{x}\le 1+x$ and the final one is due to (i.b) in Assumption \ref{ass:gen1}. 
For the third term on the right-hand side of \eqref{eq:L1normtheta}, denoting by $\widetilde W^j$ the $j$-th coordinate of the Brownian motion $\widetilde W$ we have
\begin{equation}\label{eq:bndtheta0_away}
\begin{aligned}
\E_x\Big[\sup_{0\le \lambda\le s}\big|\gamma\widetilde{W}_\lambda\big|_d\Big]
&\leq 
d\gamma \E\Big[\sup_{0\le \lambda\le s}|\widetilde W^1_\lambda|\Big]
\le d\gamma \E\Big[1 + \sup_{0\le \lambda\le s}(\widetilde W^1_\lambda)^2\Big]\\
&\le d\gamma \big(1 + 4\E\big[(\widetilde W^1_s)^2\big]\big)= d\gamma(1+4s),
\end{aligned}
\end{equation}
where we use $|x|\le 1+x^2$ for the second inequality and Doob's maximal inequality for martingales for the third inequality.
Finally, by the definition of $\cA^{opt}_{t,x}$ in Lemma \ref{lem:L1bnd} we have a bound for the last term on the right-hand side of \eqref{eq:L1normtheta}:
\begin{align*}
\E_x\Big[\sup_{\lambda\in[0,s]}\Big|\int_0^sn_r\ud\nu_r\Big|_d\Big]\leq \E_x[\nu_s]\leq K_2(1+|x|_d).
\end{align*}

We combine the above five bounds to obtain
\begin{align*}
\E_x\Big[\sup_{\lambda\in[0,s]}|X_{\lambda}^{\gamma}|_d\Big]\leq&\, c_1\big(1+|x|_d\big)+c_2\int_0^s\E_x\Big[\sup_{r\in[0,\lambda]}|X_{r}^{\gamma}|_d\Big]\ud \lambda,
\end{align*}
for some constants $c_1,c_2>0$ depending on $T, D_1, D_2, d$ and $K_2$.
By Gronwall’s lemma and setting $s=T-t$, we get 
\begin{align*}
\E_x\Big[\sup_{\lambda\in[0,T-t]}|X_{\lambda}^{\gamma}|_d\Big]\leq K_4(1+|x|_d) ,
\end{align*}
for a suitable $K_4>0$, which concludes the proof.
\end{proof}

\subsection{Convergence as \texorpdfstring{$\gamma\to 0$}{gamma->0}}\label{sec:gamma->0}
Assumptions \ref{ass:gen1} and \ref{ass:gen3} hold throughout this section.
\begin{theorem}\label{thm:convugam}
The pointwise limit $u\coloneqq\lim_{\gamma\to0}u^\gamma$ exists on $\R^{d+1}_{0,T}$. Moreover, $u$ coincides with the value of the game with payoff \eqref{eq:payoff}, i.e., $u=\underline{v}=\overline{v}=v$. Furthermore, there exists $C>0$ such that
\begin{align}\label{eq:convunifuthetau}
|u^\gamma(t,x)-v(t,x)|\leq C \gamma \quad\text{for all $(t,x)\in\R^{d+1}_{0,T}$.}
\end{align}
\end{theorem}

\begin{proof}
Let $u^{\gamma}$ be the value of the game from Theorem \ref{thm:from3to5} and set 
\[
\underline{u}\coloneqq \liminf_{\gamma\to0}u^\gamma\quad\text{ and }\quad\overline{u}\coloneqq \limsup_{\gamma\to0}u^\gamma.
\]
We are going to show that 
\begin{align*}
\overline{u}(t,x)\leq \underline{v}(t,x) \quad\text{and}\quad \underline{u}(t,x)\geq \overline{v}(t,x)
\end{align*}
for all $(t,x)\in\R^{d+1}_{0,T}$, so that $\underline{u}=\overline{u}=\underline{v}=\overline{v}=v$ as claimed.

Let us first prove that $\underline{u}\geq \overline{v}$. Fix $(t,x)\in\RdT$ and $\eta > 0$.  Let $(n,\nu)\in\cA_t$ be an $\eta$-optimal control for $u^\gamma(t,x)$, in the sense that $u^\gamma(t,x)\ge \sup_{\sigma\in\cT_t}\cJ^\gamma_{t,x}(n,\nu,\sigma)-\eta$; bear in mind that $(n,\nu)$ depends on $\gamma$. Then, $\overline v(t,x)\le \sup_{\sigma\in\cT_t}\cJ_{t,x}(n,\nu,\sigma)$ and we can pick a stopping time $\tau\in\cT_t$ such that $\sup_{\sigma\in\cT_t}\cJ_{t,x}(n,\nu,\sigma)\le \cJ_{t,x}(n,\nu,\tau)+\eta$; bear in mind that $\tau$ depends on $(n,\nu)$ and $\eta$.
Recall that the processes $X^{[n,\nu],\gamma}$ and $ X^{[n,\nu]}$ are solutions of \eqref{eq:SDEgamma} and \eqref{eq:SDE}, respectively. 
Then
\begin{align}\label{eq:splitbetalip}
u^{\gamma}(t,x)-\overline{v}(t,x)\notag\geq&\, \cJ_{t,x}^{\gamma}(n,\nu,\tau)-\cJ_{t,x}(n,\nu,\tau)-2\eta\notag\\
=&\,\E_x\Big[e^{-r\tau}\big(g(t+\tau,X_{\tau}^{[n,\nu],\gamma})-g(t+\tau,X_{\tau}^{[n,\nu]})\big)\notag\\
&\qquad+\int_0^\tau\!e^{-rs}\big(h(t+s,X_{s}^{[n,\nu],\gamma})-h(t+s,X_{s}^{[n,\nu]})\big)\,\ud s\Big]-2\eta\\
\geq &\,-K\E_x\Big[\big|X_{\tau}^{[n,\nu],\gamma}-X_{\tau}^{[n,\nu]}\big|_d+\int_0^{T-t}\big|X_{s}^{[n,\nu],\gamma}-X_{s}^{[n,\nu]}\big|_d\,\ud s\Big]-2\eta\notag\\
\geq &\,-K(1+T)\E_x\Big[\sup_{s\in[0,T-t]}\big|X_{s}^{[n,\nu],\gamma}-X_{s}^{[n,\nu]}\big|_d\Big]-2\eta,\notag
\end{align}
where $K>0$ is the Lipschitz constant from \eqref{eq:ghsmooth}. By Proposition \ref{prop:Lipschitzproc} we have the following bound:
\begin{equation}\label{eqn:th44a}
u^{\gamma}(t,x)-\overline{v}(t,x)\geq -K(1+T)K_3\gamma-2\eta.
\end{equation}
Taking $\liminf$ as $\gamma\to0$ on both sides, we get $\underline{u}(t,x)-\overline{v}(t,x)\geq-2\eta$. Recalling that $\eta > 0$ is arbitrary, we conclude that $\underline{u}(t,x)\geq \overline{v}(t,x)$ as claimed.

The proof of the inequality $\overline{u}\leq \underline{v}$ follows similar lines. Fix $(t,x) \in \RdT$ and $\eta > 0$. Let $\tau\in\cT_t$ be an $\eta$-optimal stopping time for $u^{\gamma}$, in the sense that $u^\gamma(t,x)\le \inf_{(n,\nu)\in\cA_t}\cJ^\gamma_{t,x}(n,\nu,\tau)+\eta$ and $(n,\nu)\in\cA^{opt}_{t,x}$ be such that $\inf_{(n',\nu')\in\cA_t}\cJ_{t,x}(n',\nu',\tau)\ge \cJ_{t,x}(n,\nu,\tau)-\eta$. Since $u^\gamma(t,x)\le \cJ^\gamma_{t,x}(n,\nu,\tau)+\eta$ and $\underline{v}(t,x) \ge \cJ_{t,x}(n,\nu,\tau)-\eta$, we can repeat the estimates above and obtain
\begin{equation}\label{eqn:th44b}
u^{\gamma}(t,x)-\underline{v}(t,x)\leq K(1+T)K_3\gamma+2\eta.
\end{equation}
Taking $\limsup$ as $\gamma\downarrow0$ and thanks to the arbitrariness of $\eta$ we get $\overline{u}(t,x)\leq \underline{v}(t,x)$. Similarly, inequalities \eqref{eqn:th44a} and \eqref{eqn:th44b} imply \eqref{eq:convunifuthetau}. 
\end{proof}

\begin{remark}\label{rem:uniform_conv}
The rate of convergence of $u^\gamma$ to $v$ is linear in $\gamma$ and uniform over $\RdT$. In particular, $v$ is continuous.
\end{remark}

As for $u^\gamma$, the boundedness and positivity of $f,g,h$ imply that the value function of the game $v$ is bounded.
Notice that the variational inequality \eqref{eq:varineq} implies a bound on the gradient of $u^\gamma$: $|\nabla u^\gamma|_d\leq f$ for $\gamma>0$. This, together with the above remark, yields the next corollary.
\begin{corollary}\label{cor:5.2}
The value function $v$ is is bounded on $\RdT$ and it is Lipschitz in the spatial coordinates with constant bounded by $f$, i.e., $|\nabla v(t,x)|_d\leq f(t)$ for a.e. $(t,x)\in\R^{d+1}_{0,T}$. In particular, since $f$ is non-increasing in time, we have that $v$ is Lipschitz in space with constant $f(0)$.
\end{corollary}

We turn our attention to the optimality of the stopping time $\theta_*=\tau_*\wedge\sigma_*$, where $\tau_*$ and $\sigma_*$ are defined in \eqref{eq:taustar}. Recall that these stopping times are hitting times of the underlying process so they depend on the control $(n, \nu)$. We emphasise that we work under Assumption \ref{ass:gen3}. 
\begin{lemma}\label{lem:convth}
Recall $\theta^\gamma_*$ as in Lemma \ref{lem:theta*opt} and fix $(t,x)\in\R^{d+1}_{0,T}$. For any $(n,\nu)\in\cA_t$, there is a sequence $(\gamma_k)_{k\in\N}$, converging to zero as $k\to\infty$ and possibly depending on $(n,\nu)$, for which 
\[
\liminf_{k\to\infty}\theta^{\gamma_k}_*(t, x; n,\nu)\ge \theta_*(t, x; n,\nu),\qquad\P_x-a.s.
\]
\end{lemma}

\begin{proof}
Fix $(t,x)\in\R^{d+1}_{0,T}$ and take $(n,\nu)\in\cA_t$. Thanks to Proposition \ref{prop:Lipschitzproc} there exists a sequence $(\gamma_k)\subset (0,1)$ such that 
\begin{align}\label{eq:unif}
\lim_{k\to \infty}\sup_{s\in[0,T-t]}\big|X^{[n,\nu],\gamma_k}_s-X^{[n,\nu]}_s\big|_d=0,\quad\P_x-a.s.
\end{align}
Let us denote 
\begin{align*}
&Z_s=(v-g)(t+s,X^{[n,\nu]}_s),\quad \quad Z^{k}_s=(u^{\gamma_k}-g)(t+s,X^{[n,\nu]}_s)\quad \text{and}\\
&\hat Z_s^{k}=(u^{\gamma_k}-g)(t+s,X^{[n,\nu],\gamma_k}_s).
\end{align*}
Notice that the stopping time $\theta_*^{\gamma_k}$ admits an equivalent representation as the first time that either $\hat Z^{k}_{s-}$ or $\hat Z^{k}_s$ is equal to $0$, i.e., $\theta_*^{\gamma_k}=\inf\{s\ge 0:\min\{\hat Z^k_{s-},\hat Z^k_s\}=0\}$. 

For $\omega\in\Omega$ such that $\theta_*(\omega)=0$ the claim in the lemma is trivial. Let $\omega\in\Omega$ be such that $\theta_*(\omega)>0$. Take arbitrary $\delta<\theta_*(\omega)$. Then, by the definition of $\tau_*$ and $\sigma_*$ we have
\[
\min\{Z_s(\omega),Z_{s-}(\omega)\}>0 \quad\text{for all $s\in[0,\delta]$}.
\]
As a result of Remark \ref{rem:uniform_conv}, $v-g$ is continuous, so the process $s \mapsto Z_s$ is right-continuous with left-limits. Hence, the mapping $s\mapsto \min\{Z_s(\omega),Z_{s-}(\omega)\}$ is lower semi-continuous and there exists $\lambda_{\delta,\omega}>0$ such that 
\[
\inf_{0\le s\le \delta}\min\{Z_s(\omega),Z_{s-}(\omega)\}\ge\lambda_{\delta,\omega}.
\]
Uniform convergence of $u^\gamma$ to $v$ (Theorem \ref{thm:convugam}) yields
\begin{align*}
\lim_{k\to \infty}\sup_{0\le s\le \delta}\big(|Z^k_s(\omega)-Z_s(\omega)|+|Z^k_{s-}(\omega)-Z_{s-}(\omega)|\big)=0.
\end{align*}
Moreover, Lipschitz continuity of $u^\gamma$ and $g$ (recall that $|\nabla u^\gamma|_d\le f(0)$ and $|\nabla g|_d\le f(0)$) and the convergence \eqref{eq:unif} give
\begin{align*}
\lim_{k\to \infty}&\sup_{0\le s\le \delta}\big(|\hat Z^{k}_s(\omega)-Z_s^{k}(\omega)|+|\hat Z^{k}_{s-}(\omega)-Z_{s-}^k(\omega)| \big)\\
& \leq 2 f(0)\lim_{k\to\infty}\sup_{0\le s\le \delta}\big(|X^{[n,\nu],\gamma_k}_s(\omega)-X_s^{[n,\nu]}(\omega)|_d+|X^{[n,\nu],\gamma_k}_{s-}(\omega)-X_{s-}^{[n,\nu]}(\omega)|_d\big)=0.
\end{align*}
Hence, for all sufficiently large $k$ (so all small enough $\gamma_k$) we have
\[
\inf_{0\le s\le \delta}\min\{\hat Z^{k}_s(\omega),\hat Z^{k}_{s-}(\omega)\}\ge\frac{\lambda_{\delta,\omega}}{2},
\]
which implies 
\[
\liminf_{k\to\infty}\theta_*^{\gamma_k}(\omega)\ge \delta.
\]
By the arbitrariness of $\delta$, we conclude that $\liminf_{k\to\infty}\theta_*^{\gamma_k}(\omega)\ge\theta_*(\omega)$.
\end{proof}

An adaptation to our setting of arguments from \cite[Thm.\ 4.4]{bovo2023degcontrols} allows us to prove the optimality of the stopping time $\theta_*$.
\begin{theorem}\label{thm:opttaustar}
For any $(t,x)\in\R^{d+1}_{0,T}$, we have 
\[
v(t,x)=\inf_{(n,\nu)\in\cA_{t}}\cJ_{t,x}\big(n,\nu,\theta_*(t, x; n,\nu)\big),
\]  
hence $\theta_*$ is optimal for the stopper in the game with value $v$.
\end{theorem}
\begin{proof}
We fix $(n,\nu)\in\cA_{t}$ arbitrarily but independent of $\gamma$. By standard verification arguments for $u^\gamma$ (see the first part of the proof of \cite[Thm.\ 4.4]{bovo2023degcontrols} replacing $f$ and $X^{[n,\nu]}$ therein with $f^\gamma$ and $X^{[n,\nu],\gamma}$, respectively) and the optimality of $\theta^\gamma_*=\theta^\gamma_*(t, x; n,\nu)$ for the stopper we have
\begin{align*}
u^\gamma(t,x) \le&\, \E_x\Big[e^{-r(\theta^\gamma_*\wedge\theta_*)}u^\gamma\big(t+\theta_*^\gamma\wedge\theta_*,X_{\theta_*^\gamma\wedge\theta_*-}^{[n,\nu],\gamma}\big)\!+\!\int_{0}^{\theta_*^\gamma\wedge\theta_*}\!\!\!e^{-rs}h(t+s,X_s^{[n,\nu],\gamma})\ud s \notag\\
&\qquad+\int_{[0,\theta_*^\gamma\wedge\theta_*)}\!\!e^{-rs}f(t+s)\,\ud \nu_s\Big].
\end{align*}
Recall that $|u^\gamma- v|(t,x) \le C \gamma$ by Theorem \ref{thm:convugam} and $v$ is Lipschitz in the space variable with constant $f(0)$, by Corollary \ref{cor:5.2}. By assumption, $h$ is Lipschitz with the constant $K$, see \eqref{eq:ghsmooth}. Using those properties in the above estimate for $u^\gamma$, we obtain
\begin{equation}\label{eqn:u_gamma_est}
\begin{aligned}
u^\gamma(t,x) \le&\,C\gamma\!+\!\big(f(0)\!+\!TK\big)\E_x\Big[\sup_{s\in[0,T-t]}\big|X^{[n,\nu],\gamma}_s\!-\!X^{[n,\nu]}_s\big|_d\Big]\\
&+\! \E_x\Big[e^{-r(\theta^\gamma_*\wedge\theta_*)}v\big(t+\theta_*^\gamma\wedge\theta_*,X_{\theta_*^\gamma\wedge\theta_*-}^{[n,\nu]}\big)\!+\!\int_{0}^{\theta_*^\gamma\wedge\theta_*}\!\!\!\!e^{-rs}h(t+s,X_s^{[n,\nu]})\ud s\\
&\quad\qquad\!+\!\int_{[0,\theta_*^\gamma\wedge\theta_*)}\!\!\!e^{-rs}f(t\!+\!s)\,\ud \nu_s\Big].
\end{aligned}
\end{equation}
We now let $(\gamma_k)$ be the sequence converging to zero from Lemma \ref{lem:convth} so that $\liminf_{k \to \infty} \theta_*^{\gamma_k}\wedge\theta_* =\theta_*$. In combination with an obvious bound $\limsup_{k \to \infty} \theta_*^{\gamma_k}\wedge\theta_* \le \theta_*$, this yields the limit $\lim_{k \to \infty} \theta_*^{\gamma_k}\wedge\theta_* = \theta_*$. 
Since the mappings 
\[
s\mapsto X^{[n,\nu]}_{s-}\quad\text{and}\quad s\mapsto \int_{[0,s)}e^{-ru}f(t+u)\ud \nu_u
\]
are left-continuous $\P_x$-a.s.\ and $\theta_*^{\gamma_k}\wedge\theta_*$ converges to $\theta_*$ from below (although not strictly from below), we can conclude that for a.e.\ $\omega\in\Omega$
\begin{align*}
\lim_{k\to \infty}X^{[n,\nu]}_{\theta_*^{\gamma_k}\wedge\theta_*-}=X^{[n,\nu]}_{\theta_*-}\quad\text{and}\quad\lim_{k \to \infty}\int_{[0,\theta_*^{\gamma_k}\wedge\theta_*)}e^{-rs}f(t+s)\ud \nu_s=\int_{[0,\theta_*)}e^{-rs}f(t+s)\ud \nu_s.
\end{align*}
The boundedness of $g, h, f$ and $v$ (cf.\ Corollary \ref{cor:5.2}), and $\E_x[\nu_{T-t}] < \infty$ allow us to use the dominated convergence theorem in \eqref{eqn:u_gamma_est} to obtain
\begin{align}\label{eq:vup}
v(t,x)\le \E_x\Big[e^{-r\theta_*}v\big(t+\theta_*,X_{\theta_*-}^{[n,\nu]}\big)\!+\!\int_{0}^{\theta_*}\!\!\!e^{-rs}h(t+s,X_s^{[n,\nu]})\ud s +\int_{[0,\theta_*)}\!\!e^{-rs}f(t+s)\,\ud \nu_s\Big],
\end{align}
where we also used Proposition \ref{prop:Lipschitzproc} to see that the second term of \eqref{eqn:u_gamma_est} converges to $0$.

On the event $\{\sigma_*<\tau_*\}$ it holds 
\[
v\big(t+\theta_*,X_{\theta_*-}^{[n,\nu]}\big) = g\big(t+\sigma_*,X_{\sigma_*-}^{[n,\nu]}\big) \le g\big(t+\sigma_*,X_{\sigma_*}^{[n,\nu]}\big)+f(t+\sigma_*)\Delta\nu_{\sigma_*},
\]
since $|\nabla g(t+s, \cdot)|_d \le f(t+s)$.
On the event $\{\sigma_*\ge \tau_*\}$, the bound $|\nabla v(t+s, \cdot)|_d \le f(t+s)$ yields
\[
v\big(t+\theta_*,X_{\theta_*-}^{[n,\nu]}\big)\le v\big(t+\tau_*,X_{\tau_*}^{[n,\nu]}\big)+f(t+\tau_*)\Delta\nu_{\tau_*} = 
g\big(t+\tau_*,X_{\tau_*}^{[n,\nu]}\big)+f(t+\tau_*)\Delta\nu_{\tau_*}.
\]
Substituting the above bounds in the right-hand side of \eqref{eq:vup} yields
\[
v(t,x)\le \cJ_{t,x}(n,\nu,\theta_*).
\]
By the arbitrariness of $(n,\nu)\in\cA_{t}$ we deduce optimality of $\theta_*$.
\end{proof}

\subsection{Relaxing Assumption \ref{ass:gen3} into Assumption \ref{ass:gen2}}\label{sec:ass3.2}
Despite a different setting of the game, the arguments from Section 4.2 in \cite{bovo2023degcontrols} can be repeated verbatim. We will only provide main ideas and refer the reader to detailed proofs in the aforementioned paper. The $d$-dimensional {\em open} ball centred in $0$ with radius $k$ is denoted by $B_k$. We approximate functions $f$, $g$, $h$ (satisfying Assumption \ref{ass:gen2}) with smooth bounded functions $f^{j,k}_m$, $g^{j,k}_m$ and $h^{j,k}_m$ for $j,k,m\in\N$. The index  $m$ corresponds to the truncation of the function values by $m$ (i.e., $f_m=f\wedge m$, $g_m=g\wedge m$, $h_m=h\wedge m$; recall that $f,g,h$ are non-negative), the index $j$ corresponds to the mollification by convolution with a suitable mollifier $\zeta_j$, and the index $k$ refers to the support (the functions are forced to be zero outside $B_k$ by multiplication with a cut-off function). The approximating functions are such that $(f^{j,k}_m,g^{j,k}_m,h^{j,k}_m)\to(f,g,h)$ uniformly over compact subsets of $\R^{d+1}_{0,T}$, as $j\to\infty$, $k\to\infty$ and $m\to\infty$ in this order. For each treble $(f^{j,k}_m,g^{j,k}_m,h^{j,k}_m)$ we define an approximating game which has a value $v^{j,k}_m$ by Theorem \ref{thm:convugam}.

When passing to the limit in $v^{j,k}_m$ as $j\to\infty$, $k\to\infty$ and $m\to\infty$ in this order, we use two main ingredients: the strict sub-linear growth of $g$ and $h$ and the estimates from Lemma \ref{lem:nubnd} and Lemma \ref{lem:stabL1res}. In particular, we use that for any $(n,\nu)\in\cA_{t,x}^{opt}$ and any stopping time $\theta\in \cT_t$
\begin{align}
\P_x(X_\theta^{[n,\nu]} \notin B_k)\leq \frac{1}{k}\E_x\big[|X_\theta^{[n,\nu]}|_d\big]\leq \frac{K_4(1+|x|_d)}{k}.
\end{align}

Set
\[
\overline v_\infty\coloneqq \limsup_{m\to\infty}\limsup_{k\to\infty}\limsup_{j\to\infty}v^{j,k}_m,
\qquad\text{and}\qquad
\underline v_\infty\coloneqq \liminf_{m\to\infty}\liminf_{k\to\infty}\liminf_{j\to\infty}v^{j,k}_m.
\]
\begin{proposition}\label{prop:v_infty}
Let Assumptions \ref{ass:gen1} and \ref{ass:gen2} hold. For any $(t,x)\in\R^{d+1}_{0,T}$ we have
\begin{align*}
\underline v_\infty(t,x)=\overline v_\infty(t,x)=\underline{v}(t,x)=\overline{v}(t,x),
\end{align*}
hence the value $v$ of the game \eqref{eq:valuegame} exists. Moreover, $v^{j,k}_m$ converge to $v$ uniformly on compact subsets of $\R^{d+1}_{0,T}$ and  $|\nabla v(t,x)|_d\leq f(t)$ for a.e. $(t,x)\in\R^{d+1}_{0,T}$. 
\end{proposition}

The proof is analogous to the one of \cite[Lem.\ 4.7]{bovo2023degcontrols}, with $d=d_0$ therein. We emphasise that the strict sublinear growth of $g,h$ (cf.\ \eqref{eq:ghbeta}) is needed in this proof. Finally, the gradient bound can be deduced by the fact that $|\nabla v^{j,k}_m|_d\leq f$ for all $j,k,m$.

It remains to prove the optimality of $\theta_*$ under Assumptions \ref{ass:gen1} and \ref{ass:gen2}. The nature of arguments is similar as in the proof of Lemma \ref{lem:convth} and Theorem \ref{thm:opttaustar} with details to be found in \cite[Lem.\ 4.10]{bovo2023degcontrols}. The reasoning goes along the following lines. For any treble $j,k,m$ the stopping time $\theta^{j,k,m}_*=\tau^{j,k,m}_*\wedge\sigma^{j,k,m}_*$ is optimal for the stopper in the game $v^{j,k}_m$ (c.f. \eqref{eq:taustar} with $v$ and $g$ replaced by $v^{j,k}_m$ and $g^{j,k}_m$, respectively). Similarly as in Lemma \ref{lem:convth}, using uniform convergence of $v^{j,k}_m$ to $v$, we show that for any $(t,x) \in \RdT$ and $(n,\nu)\in\cA_t$
\[
\liminf_{m\to\infty}\liminf_{k\to\infty}\liminf_{j\to\infty}\theta^{j,k,m}_*(t,x; n,\nu)\ge \theta_*(t,x; n,\nu),\quad\P_x-a.s.
\]
Arguments as in Theorem \ref{thm:opttaustar} lead to the conclusion stated formally in the next proposition.

\begin{proposition}\label{prop:opt}
Let Assumptions \ref{ass:gen1} and \ref{ass:gen2} hold. For any $(t,x)\in\R^{d+1}_{0,T}$, we have
\[
v(t,x)=\inf_{(n,\nu)\in\cA_t}\cJ_{t,x}\big(n,\nu,\theta_*(t, x; n,\nu)\big),
\]
hence $\theta_*=\theta_*(t,x;n,\nu)$ is optimal for the stopper in the game with value $v$.
\end{proposition}
We emphasise that the dependence of $\theta_*$ on $t$, $x$, $n$, $\nu$ is only through the dynamics of the process $X^{[n,\nu]}$, i.e., the stopper does not need to know the controls applied by the controller in order to execute their strategy.

We now have all the ingredients needed to prove Theorem \ref{thm:usolvar}.
\begin{proof}[\bf Proof of Theorem \ref{thm:usolvar}]
The existence of the value function $v$ and the gradient bound follow from Proposition \ref{prop:v_infty}. The continuity of $v$ is guaranteed by the continuity of the approximating functions $v^{j,k}_m$ and their uniform convergence on compact sets. Optimality of the stopping time $\theta_*$ is shown in Proposition \ref{prop:opt}. The sub-linear growth of $v$ is easily deduced upon observing that $v\ge 0$ and 
\begin{align*}
v(t,x)\le&\sup_{\tau\in\cT_t} \E_{x}\Big[e^{-r\tau}g(t\!+\!\tau,X_\tau^{[e_1,0]})\!+\!\int_0^{\tau}\! e^{-rs}h(t\!+\!s,X_s^{[e_1,0]})\,\ud s\Big]
\\
\le &\,K_1(1+T)\E_{x}\Big[\big(1+\sup_{0\le s\le T-t}\big|X_s^{[e_1,0]}\big|_d^\beta\big)\Big]\le c\big(1+|x|_d^\beta\big),
\end{align*}  
for a suitable $c>0$, where the second inequality is due to \eqref{eq:ghbeta} and the final one is due to standard estimates for uncontrolled SDEs (e.g., \cite[Cor.\ 2.5.12]{krylov1980controlled}).
\end{proof}

\section{Relaxation of growth conditions on \texorpdfstring{$g$ and $h$}{g and h}}\label{sec:ext}
So far, we studied the properties of the game under Assumption \ref{ass:gen2} on the payoff functions $g,h$ and the cost $f$. The main requirement was the continuity and the strict sub-linear growth of $g$ and $h$  (see \eqref{eq:ghbeta}), the latter being instrumental in our arguments to pass to the limit in the approximation procedure used in Proposition \ref{prop:v_infty}.

In this section we obtain an analogue of Theorem \ref{thm:usolvar} under different assumptions on $g$ and $h$. Imposing more smoothness, we can allow for $g$ to have a linear growth and $h$ to have a quadratic growth. In order to state new assumptions, we recall the H\"older space $C^{0,1,\alpha}_{\ell oc}(\RdT)$ of functions which along with their first-order spatial derivatives are $\alpha$-H\"older continuous on any compact subset of $\RdT$. Analogously, we also introduce the space $C^{1,2,\alpha}_{\ell oc}(\RdT)\subset C^{0,1,\alpha}_{\ell oc}(\RdT)$ of functions whose time derivative and the second order spatial derivatives are also $\alpha$-H\"older continuous (for details see \cite[Sec.\ 2]{bovo2022variational} and \cite[Ch.\ 3, Sec.\ 2]{friedman2008partial}).

Throughout this section we make the following assumption.
\begin{assumption}\label{ass:gen4}
The functions $f:[0,T]\to(0,\infty)$, $g,h:\RdT\to[0,\infty)$ are such that
\begin{itemize}
\item[(i)] $g\in C^{1,2,\alpha}_{\ell oc}(\RdT)$ and $h\in C^{0,1,\alpha}_{\ell oc}(\RdT)$ for some $\alpha\in(0,1)$;
\item[(ii)] $f$ is non-increasing, positive and $f^2$ is continuously differentiable on $[0,T]$;
\item[(iii)] for all $(t,x)\in\RdT$ 
\begin{align}\label{eq:g_Lip_f2}
|\nabla g(t,x)|_{d}\leq f(t).
\end{align}
\end{itemize}
Moreover, there is $K_5>0$ such that the following hold
\begin{itemize}
\item[(iv)] for all $(t,x)\in\RdT$
\begin{align}\label{eq:hgrowth}
h(t,x)\leq K_5(1+|x|_d^2);
\end{align}
\item[(v)] there exists $\beta\in(0,1)$ such that
\begin{align}\label{eqn:Lipch_h}
|h(t,x)-h(t,y)|\leq K_5(1+|x|_d+|y|_d)^\beta|x-y|_d,\quad\text{for all $t\in[0,T]$ and $x,y\in\R^{d}$;}
\end{align}
\item[(vi)] for all $0\leq s<t\leq T$ and $x\in \RdT$
\begin{align}\label{eq:ghsmooth2}
g(t,x)-g(s,x)\leq K_5(t-s)\quad\text{and}\quad h(t,x)-h(s,x)\leq K_5(t-s);
\end{align}
\item[(vii)] for all $(t,x)\in\R^{d+1}_{0,T}$
\[
(h+\partial_tg+\cL g-rg)(t,x)\geq -K_5.
\]
\end{itemize}
\end{assumption}
Notice that the continuity of $g$ and \eqref{eq:g_Lip_f2} imply
\begin{align*}
0\le g(t,x)\leq K_5(1+|x|_d),
\end{align*}
where there is no loss of generality in assuming the same constant $K_5$ as in the rest of Assumption \ref{ass:gen4}. We also notice that (iv) is redundant because it is implied by (v). Nevertheless, we leave it as stated for clarity of exposition below (in particular, it allows us to draw clear parallels to results in \cite{bovo2022variational}). Condition (vii) ensures that there is no region in the state space to which the controller (minimiser) could push the process in order to obtain arbitrarily large (negative) running gains.

Notice that Assumptions \ref{ass:gen1} and \ref{ass:gen4} imply \cite[Ass.\ 3.1 and 3.2]{bovo2022variational} and so Theorem \ref{thm:from3to5} holds with the growth bound for the value function replaced by (see also \cite[Thm.\ 3.3]{bovo2022variational})
\[
0 \le u(t,x) \le c (1 + |x|_d^2), \qquad (t,x) \in \RdT
\]
for some $c>0$. Due to relaxed growth conditions on $g, h$, the linear growth estimates from previous sections are replaced by quadratic growth. Indeed, Lemma \ref{lem:L1bnd} holds with $\cA_{t,x}^{opt}$ therein replaced with 
\begin{equation}\label{eqn:new_cA}
\cA_{t,x}^{opt}\coloneqq\big\{(n,\nu)\in\cA_t:\, \E_x[\nu_{T-t}]\leq K_2(1+|x|_d^2)\big\}
\end{equation}
due to \eqref{eq:hgrowth} (cf.~the proof of \cite[Lem.~3.1]{bovo2023degcontrols}). This amended definition of $\cA^{opt}_{t,x}$ means that Lemma \ref{lem:stabL1res} is valid with condition \eqref{eq:XgamK4} replaced by 
\begin{equation}\label{eqn:new_growth}
\E_x\big[|X_{\tau}^{[n,\nu],\gamma}|\big]\leq K_4(1+|x|_d^2), \qquad \tau \in \cT_0.
\end{equation}
Lemma \ref{lem:theta*opt} and Proposition \ref{prop:Lipschitzproc} hold without changes.

\begin{theorem}\label{thm:convugam_1}
The pointwise limit $u\coloneqq\lim_{\gamma\to0}u^\gamma$ exists on $\RdT$. Moreover, $u$ coincides with the value of the game with payoff \eqref{eq:payoff}, i.e., $u=\underline{v}=\overline{v}=v$, and there exists $C>0$ such that
\begin{align*}
|u^\gamma(t,x)-v(t,x)|\leq C (1 + |x|_d^2)^\beta\gamma \quad\text{for all $(t,x)\in\RdT$.}
\end{align*}
\end{theorem}

\begin{proof}
The proof is similar to the one of Theorem \ref{thm:convugam}; we mainly emphasise differences arising from the relaxed growth and Lipschitz assumptions. Let $u^{\gamma}$ be the value of the game from Theorem \ref{thm:from3to5}. Put $\underline{u}\coloneqq \liminf_{\gamma\to0}u^\gamma$ and $\overline{u}\coloneqq \limsup_{\gamma\to0}u^\gamma$. 

Fix $(t,x)\in\RdT$. We first show that $\underline{u}\geq \overline{v}$. Let $(n,\nu)\in\cA^{opt}_{t,x}$ be an $\eta$-optimal control for $u^\gamma(t,x)$, where $\cA^{opt}_{t,x}$ is defined in \eqref{eqn:new_cA}. Notice that $\overline v(t,x)\le \sup_{\sigma\in\cT_t}\cJ_{t,x}(n,\nu,\sigma)$. Take $\tau\in\cT_t$ such that $\sup_{\sigma\in\cT_t}\cJ_{t,x}(n,\nu,\sigma)\le \cJ_{t,x}(n,\nu,\tau)+\eta$ (notice that $\tau$ depends on $(n,\nu)$ and $\eta$). Using \eqref{eqn:Lipch_h}, instead of the second inequality in \eqref{eq:splitbetalip} we get
\begin{equation}\label{eq:splitbetalip_1}
\begin{aligned}
u^{\gamma}(t,x)-\overline{v}(t,x)
&\geq-f(0)\E_x\Big[\big|X_{\tau}^{[n,\nu],\gamma}\!-\!X_{\tau}^{[n,\nu]}\big|_d\Big]\\
&\quad-K_5\E_x\Big[\int_0^{T-t}\big(1\!+\!\big|X_{s}^{[n,\nu],\gamma}\big|_d\!+\!\big|X_{s}^{[n,\nu]}\big|_d\big)^\beta\big|X_{s}^{[n,\nu],\gamma}\!-\!X_{s}^{[n,\nu]}\big|_d\,\ud s\Big]\!-\!2\eta.
\end{aligned}
\end{equation}
We bound the first term using Proposition \ref{prop:Lipschitzproc} and the second term with H\"older's inequality:
\begin{align*}
&u^{\gamma}(t,x)-\overline{v}(t,x)\notag\\
&\geq-f(0) K_3\gamma\!-\!K_5\int_0^{T-t} \!\!
\Big( \E_x \big[1\!+\!\big|X_{s}^{[n,\nu],\gamma}\big|_d\!+\!\big|X_{s}^{[n,\nu]}\big|_d\big]\Big)^{\beta}
\Big(\E_x\Big[\big|X_{s}^{[n,\nu],\gamma}\!-\!X_{s}^{[n,\nu]}\big|_d^{\frac{1}{1-\beta}}\Big]\Big)^{1-\beta}\,\ud s
\!-\!2\eta\\
&\geq -f(0)K_3\gamma-K_5T\big(1+2K_4(1+|x|_d^2)\big)^\beta K_3^{1-\beta}\gamma-2\eta,
\end{align*}
where for the second inequality, we applied \eqref{eqn:new_growth} for the first factor and Proposition \ref{prop:Lipschitzproc} for the second factor under the integral. In conclusion, there is $c>0$ such that 
\begin{align*}
u^{\gamma}(t,x)-\overline{v}(t,x)\geq -c(1+|x|_d^2)^\beta\gamma-2\eta, \qquad (t,x) \in \RdT.
\end{align*}
Taking the $\liminf$ as $\gamma\downarrow0$ we get
\begin{align}\label{eqn:uv_up}
\underline{u}(t,x)-\overline{v}(t,x)\geq-2\eta,
\end{align}
and, by the arbitrariness of $\eta$, we conclude that $u^{\gamma}(t,x)\geq \overline{v}(t,x)$.

By following analogous arguments as above, we show that there is $c > 0$ such that
\begin{align}\label{eqn:uv_low}
u^{\gamma}(t,x)-\underline{v}(t,x)\leq&\, c(1+|x|_d^2)^\beta\gamma+2\eta, \qquad (t,x) \in \RdT.
\end{align}
Taking $\limsup$ as $\gamma\downarrow0$ and, thanks to the arbitrariness of $\eta$, we get $\overline{u}(t,x)\leq \underline{v}(t,x)$. 

Since $\underline{u} \le \overline{u}$ and $\underline{v} \le \overline{v}$, the inequalities $\overline{u}\leq \underline{v}$ and  $\underline{u}\geq \overline{v}$ imply $\underline{u}=\overline{u}=\underline{v}=\overline{v}=v$, i.e., $v$ is the value of the game.
The arbitrariness of $\eta$ in \eqref{eqn:uv_up} and \eqref{eqn:uv_low} demonstrates the validity of the estimate on the difference $u^\gamma - v$ in the statement of the theorem.
\end{proof}

The remaining results of Subsection \ref{sec:gamma->0}, i.e., Corollary \ref{cor:5.2}, Lemma \ref{lem:convth} and Theorem \ref{thm:opttaustar}, still hold. We only point out that \eqref{eqn:u_gamma_est} in the proof of Theorem \ref{thm:opttaustar} takes the following form
\begin{equation}\label{eqn:new_verif_ineq}
\begin{aligned}
u^\gamma(t,x) &\le \E_x\Big[e^{-r(\theta^\gamma_*\wedge\theta_*)}v\big(t+\theta_*^\gamma\wedge\theta_*,X_{\theta_*^\gamma\wedge\theta_*-}^{[n,\nu]}\big)\!+\!\int_{0}^{\theta_*^\gamma\wedge\theta_*}\!\!\!e^{-rs}h(t+s,X_s^{[n,\nu]})\ud s\\
&\hspace{29pt}+\int_{[0,\theta_*^\gamma\wedge\theta_*)}\!\!e^{-rs}f(t+s)\,\ud \nu_s\Big]\\
&\quad + C\E_x \big[ (1+|X_{\theta_*^\gamma\wedge\theta_*-}^{[n,\nu]}|_d^2)^\beta \big] \gamma\\
&\quad+f(0)\E_x\Big[\sup_{s\in[0,T-t]}\big|X^{[n,\nu],\gamma}_{s}-X^{[n,\nu]}_{s}\big|_d\Big]\\
&\quad+ \E_x\Big[ \int_{0}^{\theta_*^\gamma\wedge\theta_*}\!\!\!e^{-rs}\big|h(t+s,X_s^{[n,\nu], \gamma}) - h(t+s,X_s^{[n,\nu]})\big|\ud s \Big],
\end{aligned}
\end{equation}
with the last term bounded as in the proof of Theorem \ref{thm:convugam_1} by
\[
K_5T\big(1+2K_4(1+|x|_d^2)\big)^\beta K_3^{1-\beta}\gamma.
\]
Denoting $\tau = \theta_*^\gamma\wedge\theta_*$, we also note that
\begin{align*}
\E_x \big[ \big(1+\big|X_{\tau-}^{[n,\nu]}\big|_d^2\big)^\beta \big]
&\le
\E_x \big[ 1+|X_{\tau-}^{[n,\nu]}|_d^2 \big]
\le 1 + \E_x \big[\big|X_{\tau-}^{[n,\nu]} - X_{\tau-}^{[e_1,0]}\big|_d^2 \big] + \E_x \big[ \big|X_{\tau-}^{[e_1,0]}\big|_d^2 \big]\\
&\le 1 + \E [ \nu_{T-t}^2 ] + c (1 + |x|_d^2),
\end{align*}
for some constant $c > 0$, where the last inequality follows from standard growth estimates for uncontrolled SDEs with Lipschitz coefficients (compare to the proof of Lemma \ref{lem:stabL1res}). In order to conclude as in Theorem \ref{thm:opttaustar}, it remains to recall that $\E [ \nu_{T-t}^2 ] < \infty$ by the definition of $\cA_t$ and that $(n,\nu)\in\cA^{opt}_{t,x}$ is fixed arbitrarily but independently of $\gamma$.

The above changes do not affect arguments in the proof of Theorem \ref{thm:opttaustar}, so its conclusions still hold. The following theorem summarises the findings of this section.
\begin{theorem}\label{thm:last} 
The assertions of Theorem \ref{thm:usolvar} hold under Assumptions \ref{ass:gen1} and \ref{ass:gen4} with the growth bound on $v$ replaced by $ 0\leq v(t,x)\leq c(1+|x|_d^2)$, for some $c>0$.
\end{theorem}

Notice that conditions (i.a) and (i.b) from Assumption \ref{ass:gen1} are needed to establish the bound \eqref{eqn:new_growth}. We recall that it is not possible to bound the second moment of $X^{[n, \nu], \gamma}_\tau$ without controlling $\E_x[\nu^2_\tau]$ and the form of the functional $\cJ_{t,x}(n,\nu,\tau)$ allows us only to bound $\E_x[\nu_{T-t}]$ as in \eqref{eqn:new_cA} (cf.\ proof of \cite[Lem.~3.1]{bovo2023degcontrols}). We use \eqref{eqn:new_growth} in the proof of Theorem \ref{thm:convugam_1} due to the form of the non-uniform Lipschitz property \eqref{eqn:Lipch_h} of $h$. If $h$ is Lipschitz continuous (say with constant $K_5$), the second term on the right-hand side of the inequality in \eqref{eq:splitbetalip_1} simplifies to $-K_5\gamma$. Then the arguments of the proof follow without the need for \eqref{eqn:new_growth} and result in the bound
\[
 |u^\gamma(t,x) - v(t,x)| \le C \gamma
\]
for some constant $C > 0$.

The bound \eqref{eqn:new_growth} is also used to bound the last term on the right-hand side of \eqref{eqn:new_verif_ineq}, but, similarly as above, Lipschitz property of $h$ changes the bound to $c \gamma$ for some constant $c>0$. We also note that the proof of Proposition \ref{prop:Lipschitzproc} uses only the Lipschitz property of $\sigma$.

\begin{corollary}\label{cor:final}
Under Assumption \ref{ass:gen4}, if $h$ is Lipschitz continuous then Assumption \ref{ass:gen1} without conditions (i.a) and (i.b) is sufficient in order to prove the assertions of Theorem \ref{thm:last}.
\end{corollary}
\medskip

\section*{Statements and Declarations}

\noindent{\bf Acknowledgements}: The authors would like to thank an anonymous referee for careful reading and valuable feedback on the manuscript.

\noindent{\bf Funding}: A.\ Bovo and T.\ De Angelis were partially supported by EU -- Next Generation EU -- PRIN2022 (project ID: 2022BEMMLZ) {\em Stochastic control and games and the role of information}. T.\ De Angelis and J.\ Palczewski were partially supported by University of Warsaw grant IDUB-POB3-D110-003/2022 and Simons Semester ``Stochastic Modelling and Control'' at IMPAN, Warsaw, Poland in May-June 2023.

\bibliographystyle{plain}
\bibliography{Bibliography}

\end{document}